\documentclass[twocolumn,notitlepage,showpacs,10pt]{revtex4-2}
\usepackage{xcolor}
\usepackage[utf8]{inputenc}
\usepackage{bm}
\usepackage{graphicx}
\usepackage{tikz-cd}
\usepackage{dcolumn}
\usepackage{amsmath}
\usepackage{amssymb}
\usepackage{mathtools}
\usepackage{cancel}
\usepackage{hyperref}
\usepackage{academicons}
\allowdisplaybreaks
\usepackage{scalerel,stackengine}
\stackMath
\newcommand\reallywidehat[1]{%
\savestack{\tmpbox}{\stretchto{%
  \scaleto{%
    \scalerel*[\widthof{\ensuremath{#1}}]{\kern-.6pt\bigwedge\kern-.6pt}%
    {\rule[-\textheight/2]{1ex}{\textheight}}
  }{\textheight}%
}{0.5ex}}%
\stackon[1pt]{#1}{\tmpbox}%
}
\usepackage{float}

\restylefloat{figure}

\newtheorem{theorem}{Theorem}[section]
\newtheorem{corollary}[theorem]{Corollary}
\newtheorem{proposition}[theorem]{Proposition}

\newtheorem{lemma}[theorem]{Lemma}

\newenvironment{proof}[1][Proof]{\begin{trivlist}
\item[\hskip \labelsep {\bfseries #1}]}{\end{trivlist}}

\newcommand\restr[2]{{
  \left.\kern-\nulldelimiterspace 
  #1 
  \littletaller 
  \right|_{#2} 
  }}
\newcommand{\littletaller}{\mathchoice{\vphantom{\big|}}{}{}{}}

\newenvironment{remark}[1][Remark]{\begin{trivlist}
\item[\hskip \labelsep {\bfseries #1}]}{\end{trivlist}}

\newcommand{\qed}{\nobreak \ifvmode \relax \else
	\ifdim\lastskip<1.5em \hskip- \lastskip
	\hskip 0.5em plus0em minus0.5em \fi \nobreak
	\vrule height0.75em width0.5em depth0.25em\fi}

\begin{document}

\title{Instability of marginally outer trapped surfaces from initial data set symmetry}

\author{Abbas M \surname{Sherif}}
\email{abbasmsherif25@gmail.com}
\affiliation{Institute of Mathematics, Henan Academy of Sciences (HNAS), 228 Mingli Road, Zhengzhou 450046, Henan, China}

\begin{abstract}
Let $(\tilde{\Sigma},h_{ab},K_{ab})$ be an initial data set and let $x^a$ be a symmetry vector of $\tilde{\Sigma}$. Consider a MOTS $\mathcal{S}$ in $\tilde{\Sigma}$ and let the symmetry vector be decomposable along the unit normal to $\mathcal{S}$ in $\tilde{\Sigma}$, and along $\mathcal{S}$. In this note we present some basic results with regards to the stability of $\mathcal{S}$. The vector decomposition allows us to characterize the instability of $\mathcal{S}$ by the nature of the zero set of the normal component to $\mathcal{S}$ and the divergence of the component along $\mathcal{S}$. Further observations are made under the assumption of $\mathcal{S}$ having a constant mean curvature, and $\tilde{\Sigma}$ being an Einstein manifold. 
\end{abstract}

\maketitle


\section{Introduction}


Marginally outer trapped surfaces (MOTS), those for which light rays perpendicular to the surface neither converge nor diverge play a pivotal role in the formation of black holes and gravitational collapse. These objects are closely linked to trapped surfaces which were crucial in proof of the singularity theorems by Roger Penrose \cite{rp1} and later in subsequent works (see the references \cite{sen1,sen2,sen3} for detailed critique of the singularity theorems, as well as references therein). Properties of MOTS have been used extensively to prove uniqueness and topological censorship results \cite{fried1,gal1,gal2,piot1,ted1}. In particular, a notion of stability of MOTS, introduced by Andersson \textit{et al.} \cite{and1}, which characterizes when a MOTS constitutes the boundary between the trapped and untrapped regions of the spacetime (the MOTS propagates along a hypersurface -- known as a marginally outer trapped tube, abbreviated MOTT -- that bounds a black hole), has proven particularly useful in most of these works. (Several notions of MOTS stability preceded the work of Andersson \textit{et al.} (see \cite{Booth2,cao1,hay1,new1}), it is however \cite{and1} that is most widely used in more recent literature.) This notion of stability is formalized as an eigenvalue problem of an elliptic operator -- notated in the text $L_{\mathcal{S}}$ (to later be defined), where $\mathcal{S}$ denotes the MOTS -- with a generally complex spectrum of eigenvalues. The principal eigenvalue $\lambda_0$, that with the smallest real part and which is always real, characterizes stability: the MOTS is strictly stable, marginally stable or unstable if $\lambda_0>0, \lambda_0=0$ or $\lambda_0<0$.

While strictly stable MOTS are crucial to the singularity theorems since this indicates the presence of trapped surfaces, interest in the unstable ones have recently acquired momentum as they appear to be a prominent feature of the interior of black holes. The existence of an infinite number of unstable MOTS in the Schwarzschild black hole interior was established in \cite{Booth3}, obtained by a `shooting method' which shoots the MOTS equivalent of a geodesic, aptly named ``MOTSodesics'', from some specified axis. This was followed by subsequent works which identified unstable MOTS in several classes of spacetimes \cite{Booth4,Booth5}, and more recently \cite{Boothx}. The presence and role of unstable MOTS in binary black hole coalescence was further established in the remarkable works by Pook-Kolb \textit{et al.} \cite{Booth6,Booth7} where these MOTS were shown to play a role in the annihilation of the initial MOTS of two colliding black holes in the interior of the final black hole. 

Subsequently, formalizing when a MOTS can be unstable became a necessary interest. When a spacetime admits a symmetry, there are implications for the stability of MOTS. In the case of a spacetime admitting  a timelike conformal Killing vector  with conformal factor whose gradient is timelike, it was recently shown in \cite{mars3} that any embedded MOTS in such a spacetime is unstable, which was then applied to the de Sitter spacetime. (The relationship between symmetries and MOTS existence were previously considered in an early work by Carrasco and Mars \cite{mars4}.) Naturally one would like to consider the relationship between symmetry and stability purely from a local perspective, i.e. from the initial data point of view.

Consider an initial data (ID) set $(\tilde{\Sigma},h_{ab},K_{ab})$ and an embedded (or more mildly, immersed) MOTS $\mathcal{S}\subset\tilde{\Sigma}$, and suppose a vector field $x^a$ is a symmetry (to later be defined) of $\tilde{\Sigma}$. For certain characters of the vector field $x^a$ along $\mathcal{S}$, the authors in \cite{Booth1} established several instability results for the MOTS.  In particular the authors were interested in those cases where the symmetry of the slice $\tilde{\Sigma}$ is not shared by the MOTS but rather the symmetry vector is either tangent to the MOTS at some point or nowhere tangent to the MOTS. 

For a decomposable symmetry vector along the unit normal direction to $\mathcal{S}$ in $\tilde{\Sigma}$ and along $\mathcal{S}$ itself, we can project the symmetry criterion to the MOTS and conditions on the vector components will play a role in characterizing instability as one would expect. Such decomposition provides a neat covariant way to analyze stability through algebraic relations on functions defined on the MOTS. The role of the extrinsic geometry of $\tilde{\Sigma}$ is also quite key. This is related to a differential relation on the mean curvature of the MOTS. Our aim in this note is to establish some instability results, relying on few of the results of \cite{Booth1}, in terms of the character of the components of the vectors as well as how instability relates to the extrinsic curvature of the slice. We will also briefly consider how these results are modified/relaxed when some geometric restrictions are imposed on the MOTS and the embedding slice.

We organize this paper as follows. Section \ref{sec2} briefly introduces the notion of MOTS and their stability, recalling recent results that are of relevance to the current work. In Section \ref{sec3}, we introduce the symmetry vector type that is of interest here and present the results of this work. We then state and prove several basic instability results for embedded MOTS under the decomposition assumption on the symmetry vector of the embedding slice. For certain geometric restrictions on the slice and the MOTS, some comments on the nature of the instability of the MOTS are provided. We conclude with discussion of results in Section \ref{sec5}.


\section{MOTS, stability and initial data symmetry}\label{sec2}


Here we briefly review the geometry of marginally outer trapped surfaces (MOTS) in spacetimes. (One may consult the references \cite{sh1,ash1,ib11,ib12} for more details on the geometry of these objects, with more relatively recent covariant formulations adapted in \cite{el1,as1,as2,as3}.) We then introduce results on the relationship between symmetry of an initial data set containing a MOTS and its instability.

Consider a 2-surface $\mathcal{S}$ embedded in a 4-dimensional spacetime $(\mathcal{M},g_{ab})$ obeying the Einstein field equations $G_{ab}=T_{ab}$ where $G_{ab}$ is the Einstein tensor and $T_{ab}$ is the Energy momentum tensor. Throughout $\mathcal{S}$ will be assumed to be orientable. We denote the induced metric on $\mathcal{S}$ and its compatible covariant derivative operator respectively by $q_{ab}$ and $\mathcal{D}_a$. 

The normal space of $\mathcal{S}$ is spanned by a pair of null normal vector fields, which we denote by $k^a$ and $\ell^a$, and are normalized as $k^a\ell_a=-2$.

We assume a decomposition of $\mathcal{M}$ into spacelike slices $\{\Sigma_t\}$, where $t$ parametrizes the foliation along a normalized timelike direction that we denote by $u^a$. Fix a slice $\tilde{\Sigma}$ in $\{\Sigma_t\}$ and let $\mathcal{S}\subset\tilde{\Sigma}$ be a closed embedded spacelike surface, and let $n^a$ be the outward pointing normal to $\mathcal{S}$ in $\tilde{\Sigma}$. We construct the null pair as

\begin{eqnarray}
k^a=u^a+n^a,\quad \ell^a=u^a-n^a,\label{nullv}
\end{eqnarray}
chosen such that $k^a$ points outward and $l^a$ points inward to $\mathcal{S}$. The null expansions are then expressed respectively as

\begin{eqnarray}
\theta_k=\mathcal{Z}_1+\mathcal{Z}_2,\quad\theta_l=\mathcal{Z}_1-\mathcal{Z}_2,\label{mot1}
\end{eqnarray}
where we have written the divergences $\mathcal{Z}_1=\mathcal{D}_au^a$ and $\mathcal{Z}_2=\mathcal{D}_an^a$. Then, the surface $\mathcal{S}$ is a marginally outer trapped surface (MOTS) if $\theta_k$ vanishes on $\mathcal{S}$ and will be said to be a minimal MOTS if $\theta_k$ and $\theta_l$ both vanish on $\mathcal{S}$. Note that the minimal condition is equivalent to the vanishing of the mean curvature $\mathcal{Z}_2$ of $\mathcal{S}$. If there are only isolated points where both $\theta_k$ and $\theta_l$ vanish, we will say that those points are minimal points of $\mathcal{S}$.

MOTS may or may not evolve to foliate a horizon due to various obstructions. When they do evolve to a horizon, under certain physical assumptions on the spacetime such horizon will enclose trapped surfaces, i.e. the horizon bounds a black hole. The nature of the evolution itself determines whether a black hole will expand or stay in equilibrium, in which case these horizons, for a smooth evolution, carry the particular names dynamical horizons and isolated horizons (see for example the references \cite{sh1,ib11,ash1,ash2} where these objects have been extensively analyzed, including their uniqueness, and for those cases where the spacetime admits a local rotational symmetry see \cite{as1,as2,as3}).

The notion of MOTS stability introduced by Andersson \textit{et al.}, \cite{and1}, also characterizes the evolution of MOTS under certain curvature assumptions on the spacetime. It was established that a strictly stable MOTS will necessarily evolve to a smooth horizon bounding a black hole. And under the strict null curvature condition, the horizon is a dynamical horizon. The stability operator has been used in several works, analytically, to demonstrate various aspects of the stability/instability of MOTS in some simple exact solutions (see for example \cite{pm1,as4} as well as the examples addressed in \cite{Booth1}).

The stability operator is obtained as the outward deformation

\begin{eqnarray}
\delta_{\psi n}\theta_k=L_{\mathcal{S}}\psi,
\end{eqnarray}
which allows the rephrasing of the stability of the MOTS as an eigenvalue problem as previously discussed in the intrroduction. Explicitly, the operator has the form

\begin{eqnarray}
L_{\mathcal{S}}=-\arrowvert \mathcal{D}-\omega\arrowvert^2+\frac{1}{2}\left(\mathcal{R}-\arrowvert \sigma_k\arrowvert^2-2G_{ab}k^au^b\right).
\end{eqnarray}
Here $\mathcal{R}$ is the scalar curvature of the MOTS, $\sigma_k$ is the shear of $k^a$ along $\mathcal{S}$ and the 1-form $2\omega_a=-\ell_bq^c_a\nabla_ck^b$ is the connection on $T^{\perp}(\mathcal{S})$, the normal bundle of $\mathcal{S}$. This 1-form is identified with angular momentum and may be seen as a perturbation away from self-adjointness of the stability operator: when $\omega_a$ vanishes identically the operator is self-adjoint and when $\omega_a$ is a gradient of some function on $\mathcal{S}$, the operator is equivalent to a self-adjoint one, a property of the scaling freedom of the null vectors. This fact is shown in \cite{Booth1} (see also the work by Koh \textit{et al.} \cite{as5} where various aspects pertaining to the self-adjointness of the operator was discussed in a covariant setting for spacetimes adnitting a particular decomposition.)

To reiterate, the characterization of stability of $\mathcal{S}$ follows from the sign of the principal eigenvalue $\lambda_0$ where strict stability, marginal stability, or instability is given by $\lambda_0$ being strictly positive, zero, or negative.

When a MOTS is unstable, it is now understood that it cannot evolve to a black hole horizon. Unstable MOTS can be found in the Robertson walker solutions, de Sitter spacetimes and even simple static solutions as have been demonstrated by works referenced earlier. These MOTS are usually identified with nontrivial topologies, i.e. aspherical topologies including very complicated intersections. These unstable scenarios are the concentration of the present work. Our consideration will be entirely localized in the sense that, there will be no mention of a parent spacetime, but rather exclusive reference to the slice and the embedded surface. We introduce some recent results of \cite{Booth1} on the instability of a MOTS in a given initial data (ID) set with a particular symmetry. 

For a given ID set $(\tilde{\Sigma},h_{ab},K_{ab})$, we say that a non-trivial vector field $x^a$ is a symmetry of $\tilde{\Sigma}$ if the following hold:

\begin{eqnarray}
\mathcal{L}_xh_{ab}=0\quad\mbox{and}\quad\mathcal{L}_xK_{ab}=0,\label{stab1}
\end{eqnarray}
with $\mathcal{L}_x$ denoting the Lie derivative operator along the vector field $x^a$. For a surface $\mathcal{S}$ in $\tilde{\Sigma}$, if further $x^a$ is everywhere tangent to $\mathcal{S}$, $x^a$ is said to also be a symmetry of $\mathcal{S}$. This of course implies that the induced metric on $\mathcal{S}$ is Lie dragged along $x^a$. The following result of \cite{Booth1} is then of interest to this note, which we state without proof. 

\begin{theorem}[Booth \textit{et al.} \cite{Booth1}, Th. 1.6]\label{boot1}
Suppose $\mathcal{S}$ is a MOTS and $x^a$ is a symmetry of an ID set $(\tilde{\Sigma},h_{ab},K_{ab})$ but not of $\mathcal{S}$. Then $0$ is an eigenvalue of $L_{\mathcal{S}}$, with eigenfunction $x^an_a$. Moreover,

\begin{itemize}
\item[(1)] $\mathcal{S}$ is marginally stable iff $x^a$ is nowhere tangent to $\mathcal{S}$,

\item[(2)] $\mathcal{S}$ is unstable iff $x^a$ is tangent to $\mathcal{S}$ at some point.
\end{itemize}
\end{theorem}
(The norm here is taken with respect to the metric $h_{ab}$.) The proof of the result follows the establishment of various Lemmas.


\section{Symmetry and instability}\label{sec3}


Let $(\tilde{\Sigma},h_{ab},K_{ab})$ be an ID set and it will be assumed that $\tilde{\Sigma}$ is foliated by 2-surfaces which have spacelike unit normal $n^a$. Let one of these leaves be a MOTS $\mathcal{S}$ in $\tilde{\Sigma}$. We construct the null normals to $\mathcal{S}$ as in \eqref{nullv} so that the spacetime metric induces the 2-metric $q_{AB}$ on $\mathcal{S}$ as

\begin{eqnarray*}
q_{AB}e_a^Ae_b^B=g_{ab}+k_{(a}l_{b)},
\end{eqnarray*}
where $e_a^A$ are the operators which pulls back/pushes forward vectors and tensors between the spaces. Let $\tau^a$ be an everywhere tangent vector field on $\mathcal{S}$, and assume a decomposable vector field $x^a$ in $\tilde{\Sigma}$ according to

\begin{eqnarray}
x^a=\alpha n^a+\tau^a,\label{vec1}
\end{eqnarray}
with $\alpha$ being a smooth function on $\tilde{\Sigma}$. It is further assumed that $\alpha$ does not vanish identically on $\mathcal{S}$. (We clarify that we can have a symmetry of $\tilde{\Sigma}$ with a $u^a$ component which will be a more general consideration. Here we are interested in a vector that lies entirely in $\tilde{\Sigma}$.) With \eqref{vec1}, the first symmetry condition of \eqref{stab1} can be recast as

\begin{eqnarray}
0=\alpha\mathcal{L}_nh_{ab}+2\left(n_{(a}D_{b)}\alpha+D_{(a}\tau_{b)}\right),\label{can2}
\end{eqnarray}

To begin with, it is of interest to know when or when not the vector $x^a$ being a symmetry of $\tilde{\Sigma}$, is also a symmetry of the MOTS. Of course for $\tau^a=0$, $x^a$ is not a symmetry of $\mathcal{S}$ and trivially is nowhere tangent to $\mathcal{S}$. In such a case the MOTS is marginally stable by Theorem \ref{boot1}. So we consider those cases with non-trivial $\tau^a$. We present this as the following lemma:

\begin{lemma}\label{lem1}
Consider an ID set $(\tilde{\Sigma},h_{ab},K_{ab})$ and let \eqref{vec1} be a symmetry of $\tilde{\Sigma}$. Let $\mathcal{S}$ be a MOTS in $\tilde{\Sigma}$ and suppose the tangential component of $x^a$ along $\mathcal{S}$ is divergence-free. Then either $x^a$ is confined to $\mathcal{S}$ or $\mathcal{S}$ is minimal.
\end{lemma}

\begin{proof}
The projection of \eqref{can2} to the MOTS $\mathcal{S}$ gives the equation

\begin{eqnarray}
\alpha\mathcal{Z}_2+\mathcal{D}_a\tau^a=0,\label{can41}
\end{eqnarray}
from which the conclusion follows.\qed

\end{proof}

Indeed, the divergence-free condition implies and is implied by the condition that $\tau^a$ is a symmetry of the MOTS $\mathcal{S}$. Thus, if $\mathcal{S}$ is non-minimal, it follows that the symmetry of $\tilde{\Sigma}$ is inherited by $\mathcal{S}$ since $x^a$ and $\tau^a$ coincide.

That $x^a$ is a symmetry of the $\mathcal{S}$ implies \eqref{can41}. As we are interested in the case where the MOTS does not inherit the symmetry of the slice, two conditions accommodating this case will be the following:

\begin{enumerate}
\item $\mathcal{D}_a\tau^a\neq0$; 

\item $\tau^a$ is divergence-free and $\mathcal{S}$ is minimal.
\end{enumerate}

Of course it is possible that $\mathcal{S}$ is minimal without $\mathcal{Z}_2$ vanishing on all of $\tilde{\Sigma}$, and the relevance of this point will be established shortly. Let's start by noting that one can write the divergence $\mathcal{Z}_1$ as the projection to $\mathcal{S}$ of the extrinsic curvature of $\tilde{\Sigma}$:

\begin{eqnarray}
\mathcal{Z}_1=q_{ab}K^{ab}.
\end{eqnarray}
Then, taking the Lie derivative of the null expansion $\theta_k$ along $x^a$ and noting that $\mathcal{L}_xK_{ab}=0$ we can write

\begin{eqnarray}
K^{ab}\mathcal{L}_xq_{ab}=\alpha\mathcal{Z}_1'-\tau^a\mathcal{D}_a\mathcal{Z}_2,\label{ext1}
\end{eqnarray}
where the ``prime'' denotes derivative along  the $n^a$ direction.

Now, a few comments can be drawn from the following vanishing condition:

\begin{eqnarray}
\alpha\mathcal{Z}_1'-\tau^a\mathcal{D}_a\mathcal{Z}_2=0.\label{ext2}
\end{eqnarray}

If an existence of a solution to the above equation \eqref{ext2} is ruled out (existence of a ``solution'' should be interpreted here as the equation holds somewhere in the slice), then neither can $\tilde{\Sigma}$ be extrinsically flat nor $x^a$ a symmetry of $\mathcal{S}$. In the case of a totally umbilic $\tilde{\Sigma}$, that is $K_{ab}$ scales $h_{ab}$ by some function in $\tilde{\Sigma}$ (call this function $f$), \eqref{ext2} will hold everywhere in $\tilde{\Sigma}$ given that $x^a$ is a symmetry of $\tilde{\Sigma}$. Another scenario is where the equation \eqref{ext2} holds locally on $\mathcal{S}$ and the extrinsic curvature is non-zero on $\mathcal{S}$, in which case $x^a$ is a symmery of $\mathcal{S}$, ruling out the applicability of the results of \cite{Booth1}.

For an extrinsically flat slice, $\mathcal{S}$ will be minimal. Then, we see that either $\alpha$ vanishes and the symmetry vector is also a symmetry of the MOTS (in which case instability is not assessible from our considerations) or $\mathcal{Z}_1$ is constant along $n^a$. On the other hand, if the restriction $\restr{\mathcal{Z}_1'}{\mathcal{S}}$, which we refer to as the normal variation of $\mathcal{Z}_1$, vanishes somewhere along $\mathcal{S}$, then the mean curvature must be constant along the integral curves of $\tau^a$ through that point. And conversely, wherever the mean curvature is constant on $\mathcal{S}$, the normal variation of $\mathcal{Z}_1$ will vanish there.

From the above discussions we are in the position to present some results:

\begin{proposition}\label{pro1}
Consider an ID set $(\tilde{\Sigma},h_{ab},K_{ab})$ and let \eqref{vec1} be a symmetry of $\tilde{\Sigma}$. Let $\mathcal{S}$ be a MOTS in $\tilde{\Sigma}$ and suppose that the component of $x^a$ along $\mathcal{S}$ has a nowhere vanishing divergence. Then, $L_{\mathcal{S}}$ admits the zero eigenvalue with associated eigenfunction $\alpha$.
\end{proposition}

It is indeed clear that the condition that $\mathcal{D}_a\tau^a$ is nowhere vanishing ensures that $\mathcal{S}$ cannot lie in an extrinsically flat $\tilde{\Sigma}$. On the other hand if \eqref{vec1} is a symmetry of $\tilde{\Sigma}$ but not of $\mathcal{S}$, then $\mathcal{S}$ must be a minimal MOTS in an extrinsically flat $\tilde{\Sigma}$, i.e. \eqref{ext2} can be solved over $\tilde{\Sigma}$.

Notice that if it is imposed that the divergence of $\tau^a$ is nowhere vanishing as in the above result, then neither $\mathcal{Z}_2$ nor $\alpha$ has a zero on $\mathcal{S}$. Therefore, their smoothness will imply sign definiteness, i.e. $\alpha$ is sign definite by \eqref{can41}. Therefore the zero eigenvalue is the principal eigenvalue of $L_{\mathcal{S}}$. As a corollary to Proposition \ref{pro1} it then follows that 

\begin{corollary}\label{cor1}
Consider an ID set $(\tilde{\Sigma},h_{ab},K_{ab})$ and let \eqref{vec1} be a symmetry of $\tilde{\Sigma}$. Let $\mathcal{S}$ be a MOTS in $\tilde{\Sigma}$ and suppose that the component of $x^a$ along $\mathcal{S}$ has a nowhere vanishing divergence. Then, $\mathcal{S}$ is marginally stable.
\end{corollary}

Next, we may provide a characterization of the zero eigenvalue as before using \eqref{ext1}:

\begin{proposition}\label{pro2}
Consider an ID set $(\tilde{\Sigma},h_{ab},K_{ab})$ with symmetry vector \eqref{vec1}, and let $\mathcal{S}$ be a MOTS in $\tilde{\Sigma}$. If the equation \eqref{ext2} fails at some $p\in\mathcal{S}$, then $L_{\mathcal{S}}$ admits the zero eigenvalue with associated eigenfunction $\alpha$. 
\end{proposition}

Recalling that $\alpha$ is nonzero whenever $x^a$ is not tangent to $\mathcal{S}$, we can state the following instability result.

\begin{theorem}\label{th1}
Consider an ID set $(\tilde{\Sigma},h_{ab},K_{ab})$ and let \eqref{vec1} be a symmetry of $\tilde{\Sigma}$. Let $\mathcal{S}$ be a MOTS in $\tilde{\Sigma}$ with no minimal points and suppose \eqref{ext2} fails at some $p\in\mathcal{S}$. If $\mathcal{D}_a\tau^a$ vanishes somewhere on $\mathcal{S}$, $\mathcal{S}$ is unstable.
\end{theorem}

\begin{proof}
For the proof we simply need to establish that under the assumptions of the theorem $x^a$ is tangent to $\mathcal{S}$ somewhere (of course $x^a$ cannot be everywhere tangent since it is not a symmetry of $\mathcal{S}$ by virtue of the fact that \eqref{ext2} admits no solution), i.e. that $\alpha$ vanishes somewhere. The equation \eqref{ext2} admiting no solution ensures $x^a$ is not a symmetry of $\mathcal{S}$ (and of course $\tilde{\Sigma}$ cannot be extrinsically flat). Since $\mathcal{S}$ has no minimal points, where $\mathcal{D}_a\tau^a$ vanishes on $\mathcal{S}$, we must have $\alpha$ vanishing there.\qed
\end{proof}

Furthermore, notice that if the normal variation of $\mathcal{Z}_1$ does not vanish, then $\alpha$ on $\mathcal{S}$ will have a zero whenever the extrinsic curvature vanishes somewhere on $\mathcal{S}$ provided that $\tau^a\mathcal{D}_a\mathcal{Z}_2$ vanishes there. This implies the following instability result:

\begin{theorem}\label{th2}
Consider an ID set $(\tilde{\Sigma},h_{ab},K_{ab})$, let $\mathcal{S}$ be a MOTS in $\tilde{\Sigma}$, and suppose \eqref{vec1} is a symmetry of $\tilde{\Sigma}$ but not of $\mathcal{S}$. Let $p\in\mathcal{S}$ be a point such that $Z_2$ is constant along all points of integral curves of $\tau^a$ through $p$ where the extrinsic curvature vanishes. If the normal variation of $\mathcal{Z}_1$ is non-vanishing, $\mathcal{S}$ is unstable.
\end{theorem}

Finally we point out a condition on the mean curvature of $\mathcal{S}$ when the surface of interest is a ``genuine'' surface. That is, when the Lie bracket of $u^a$ and $n^a$ vanishes in addition to $\mathcal{D}_a\mathcal{D}_b$ commuting when acting on scalars on $\mathcal{S}$.

Taking the Lie derivative of equation \eqref{can41} along $x^a$ we have

\begin{align}
0&=\alpha(\alpha\mathcal{Z}_2)'+\tau^a\mathcal{D}_a(\alpha\mathcal{Z}_2)+\mathcal{L}_x(\mathcal{D}_b\tau^b)\nonumber\\
&=\alpha(\alpha\mathcal{Z}_2)'+\tau^a\mathcal{D}_a(\alpha\mathcal{Z}_2+\mathcal{D}_b\tau^b),\label{lie2}
\end{align}
where the additional term $(\mathcal{D}_a\tau^a)'$ of $\mathcal{L}_x(\mathcal{D}_b\tau^b)$ vanishes since the divergence of $\tau^a$ is defined entirely on the surface. Suppose that $\alpha$ is nowhere zero. Since \eqref{can41} holds on the MOTS $\mathcal{S}$, the second derivative in the second term of the second equation of \eqref{lie2} is zero and \eqref{lie2} reduces to

\begin{eqnarray}
(\alpha\mathcal{Z}_2)'=0.\label{lie3}
\end{eqnarray}
By applying the product $n^an^b$ to \eqref{can2} we find

\begin{eqnarray}
\alpha'=\tau^an_a'.
\end{eqnarray}

In the case of a genuine surface and not a pseudo-surface, necessarily the restriction of $n_a'$ to $\mathcal{S}$ vanishes. Thus $\restr{\alpha'}{\mathcal{S}}=0$. By \eqref{lie3}, it follows that $\restr{\mathcal{Z}_2'}{\mathcal{S}}=0$. That is, the variation of the mean curvature of $\mathcal{S}$ along the unit normal vanishes. So, if one starts with a minimal $\mathcal{S}$, this property is maintained if one deforms $\mathcal{S}$ along $n^a$.

Now, because $\mathcal{S}$ is closed, upon integrating \eqref{can41} over $\mathcal{S}$ to get

\begin{eqnarray}
\int_{\mathcal{S}}\alpha\mathcal{Z}_2=0.\label{integ}
\end{eqnarray}
If there are no minimal points of $\mathcal{S}$, $\mathcal{Z}_2$ has a definite sign. Thus $\alpha$ must change sign over $\mathcal{S}$, i.e. $\alpha$ must vanish somewhere, and therefore, the requirement on the divergence of $\tau^a$ in Theorem \ref{th1} can be dropped as the divergence will vanish where $\alpha$ does. 

\begin{remark}[Remark.]
If $\alpha$ is nowhere zero, then by \eqref{integ} $\mathcal{S}$ has at least one minimal point. Thus, by \eqref{can41}, results that rely on the nowhere vanishing of the divergence, for example Proposition \ref{pro1} and Corollary \ref{cor1}, are not applicable.
\end{remark}

Let it be pointed out that since $\mathcal{S}$ is closed $\mathcal{D}_a\mathcal{Z}_2=0$ at some point $p\in\mathcal{S}$. In particular, $\mathcal{S}$ is a topological 2-sphere. So, for a genuine surface, \eqref{ext2} admitting no solution on $\mathcal{S}$ requires $\alpha\neq0$ at that point $p$ (and similarly for the normal variation of $\mathcal{Z}_1$).

Suppose the normal variation of $\mathcal{Z}_1$ is nonvanishing. At $p\in\mathcal{S}$ where $\mathcal{D}_a\mathcal{Z}_2=0$, $\alpha$ will indeed vanish there if the extrinsic curvature vanishes there. This suggests the following result:

\begin{theorem}\label{th2}
Consider an ID set $(\tilde{\Sigma},h_{ab},K_{ab})$, let $\mathcal{S}$ be a MOTS in $\tilde{\Sigma}$, and suppose \eqref{vec1} is a symmetry of $\tilde{\Sigma}$ but not of $\mathcal{S}$. If the normal variation of $\mathcal{Z}_1$ is non-vanishing and $\tilde{\Sigma}$ is geodesic where $\mathcal{Z}_2$ is extremal, $\mathcal{S}$ is unstable.
\end{theorem}

We now present a few comments for the cases of certain geometric restrictions on the MOTS.

\subsection*{Constant mean curvature MOTS:}

If a MOTS is of constant mean curvature (CMC) one has that

\begin{eqnarray}
\alpha\mathcal{Z}_1'=K^{ab}\mathcal{L}_xq_{ab}.\label{imp1}
\end{eqnarray}
If the vector $u_a'$ points along $n^a$, the right hand side vanishes: directly expand the RHS of \eqref{imp1} using $q_{ab}=h_{ab}-n_an_b$, which involves the Lie derivative of $\mathcal{K}$ (which is zero) as well as terms containing the projection $K_{ab}n^b$. Explicitly, this projection is

\begin{align}
\alpha\mathcal{Z}_1'&=-2K^{ab}n_b\mathcal{L}_xn_a\nonumber\\
&=-2\left(q^{ab}u_b'-(u^bn_b')n^a\right)\mathcal{L}_xn_a\nonumber\\
&=-2u_{\bar{a}}' \mathcal{L}_xn^a,\label{proj}
\end{align}
where the bar on the index $a$ indicates a projection to $\mathcal{S}$. which when evaluated at $\mathcal{S}$ vanishes by the assumption on $u_a'$. And in this case, if $\alpha$ is nowhere zero, then the normal variation of $\mathcal{Z}_1$ evaluated at $\mathcal{S}$ vanishes. Clearly, in order for any of the results in this work to apply, the normal variation of $\mathcal{Z}_1$ cannot be nowhere vanishing as $\alpha$ will vanish identically on $\mathcal{S}$.

\subsection*{MOTS embedded in an Einstein slice:}

To conclude, let us consider a particular restriction on the geometry of the slice. If $\tilde{\Sigma}$ is an Einstein manifold so that $R_{ab}=\lambda h_{ab}$, and noting the decomposition $q_{ab}=h_{ab}-n_an_b$, we use Gauss' equation to establish that

\begin{eqnarray}
\mathcal{R}=\frac{R}{3}+\left(\mathcal{Z}_2^2-\arrowvert y\arrowvert^2\right),\label{gauss1}
\end{eqnarray}
where $\mathcal{R}$ denotes the scalar curvatures of the slice and $y$ is the covariant 2-tensor with components (it is fully projected to $\mathcal{S}$)

\begin{eqnarray*}
y_{ab}=\mathcal{D}_an_b.
\end{eqnarray*}
From the above we may also infer the following statements regarding the stability and the topology of the MOTS:

\begin{enumerate}

\item Any minimal MOTS $\mathcal{S}$ embedded in an Einstein slice of nonpositive scalar curvature is necessarily unstable following from the fact that $\mathcal{R}\leq0$.

\item Given a positive scalar curvature slice, whenever $\mathcal{S}$ negligibly shears along $n^a$, and $n^a$ does not ``twist'' i.e. has vanishing vorticity, $\mathcal{S}$ carries a strictly positive scalar curvature. This follows from the fact that under the assumption the parenthesized term of \eqref{gauss1} approximates to $\mathcal{Z}_2^2$. The same conclusion of course follows if the slice has a nonnegative scalar curvature and the MOTS has no minimal points.On the other hand, under the same assumptions imposing $\mathcal{R}\gg \mathcal{Z}_2$ would ensure that $\tilde{\Sigma}$ is positively curved.

\end{enumerate}

Additionally, for a MOTS with constant scalar curvature, we have the following condition holding on the MOTS:

\begin{eqnarray}
2\mathcal{Z}_2\mathcal{D}_a\mathcal{Z}_2-\mathcal{D}_a\arrowvert y\arrowvert^2=0.\label{vanc1}
\end{eqnarray}
Thus if the 2-tensor $y$ has a constant norm over $\mathcal{S}$, this imposes certain mild geometric restrictions on $\mathcal{S}$: $\mathcal{S}$ is minimal or has CMC. One simple example where the second term of \eqref{vanc1} vanishes is in the space of nontwisting locally rotationally symmetric (LRS) solutions where the two terms of \eqref{vanc1} coincide. Conversely, a minimal MOTS or one with constant mean curvature will necessarily have a constant $y$ norm.


\section{Summary}\label{sec5}


The stability of MOTS has consequences for the existence and evolution of a black hole. Several recent works have uncovered `exotic' MOTS in the interior of certain static black holes and merged binary black holes \cite{Booth3,Booth4,Booth5,Booth7}. These MOTS have been found to have nonspherical, and in many cases, self-intersecting topologies. One thing that these MOTS have in common is that they are unstable and therefore do not evolve to enclose trapped surfaces. When a spacetime admits an infinitesimal conformal Killing symmetry, several results on stability were obtained by the authors in \cite{mars4}. Recently it has been demonstrated that for a spacetime admitting a conformal Killing symmetry, and with an additional assumption on the gradient of the divergence of the generator of the symmetry, any MOTS in such a spacetime will necessarily be unstable \cite{mars3}. An interest has arisen in the case where an initial data set with an intrinsic Killing symmetry admits a MOTS \cite{Booth1}. And given certain tangency condition of the symmetry vector on the MOTS, the authors of \cite{Booth1} were able to obtain several instability results as well as results related to the character of the eigenvalue spectrum of the stability operator. In this note we have assumed a decomposition of the symmetry vector field along the unit normal to the MOTS in the slice, and a component along the MOTS. The instability of the MOTS is analyzed from algebraic equations on the MOTS. Instability results are then phrased in terms of the zeros of the component normal to the MOTS and the divergence of the component along the MOTS. The MOTS being closed provides some simplified criteria to guarantee instability of the MOTS. Several observations are also made under certain assumptions on the geometry of the MOTS and the slice in which it is embedded. Some comments are provided under the assumption that the MOTS has a constant mean curvature. When the slice embedding the MOTS is an Einstein manifold, using the Gauss' equation to relate the scalar curvatures of the slice and the MOTS we also draw conclusions about certain geometric features of the MOTS.

This work adds to the developing literature on the relationship between symmetries and the stability/instability of MOTS, and may provide some practical tools to directly check instability.


\section{Acknowledgement}


The author is enormously grateful to the anonymous referees for pointing out several imprecise statements and errors in the manuscript, the fixing of which has significantly improved the readibility of the manusript. The author acknowledges that this research is supported by the Institute of Mathematics, funded through the High-level Talent Research Start-up Project Funding of the Henan Academy of Sciences (Project No.: 251819085).

\end{document}